\documentclass[11pt]{amsart}

\title[Distinguishing 2-knots admitting circle actions]{Distinguishing 2-knots admitting circle actions by fundamental groups}

\author[Fukuda]{Mizuki Fukuda}
\address{Mathematics for Advanced Materials Open Innovation Laboratory, AIST, c/o AIMR, Tohoku University, 2-1-1, Katahira, Aoba-ku, Sendai, Miyagi 980-8577, Japan}
\email{mizuki.fukuda.d2@tohoku.ac.jp}

\author[Ishikawa]{Masaharu Ishikawa}
\address{Department of Mathematics, Hiyoshi Campus, Keio University, 
4-1-1 Hiyoshi, Kohoku, Yokohama 223-8521, Japan}
\email{ishikawa@keio.jp}

\usepackage{amsfonts,amsmath,amssymb,amscd}
\usepackage{amsthm}
\usepackage{latexsym}
\usepackage{graphicx}
\usepackage{psfrag}
\usepackage{color}
\usepackage{ulem} 

\usepackage{fancybox, ascmac}
\usepackage{multicol}

\theoremstyle{plain}
\newtheorem*{theorem*}{Theorem}
\newtheorem*{lemma*} {Lemma}
\newtheorem*{corollary*} {Corollary}
\newtheorem*{proposition*}{Proposition}
\newtheorem*{conjecture*}{Conjecture}
\newtheorem{theorem}{Theorem}[section]
\newtheorem{lemma}[theorem]{Lemma}

\newtheorem{proposition}[theorem]{Proposition}

\theoremstyle{remark}

\newtheorem{remark}[theorem]{Remark}

\newtheorem*{example*}{Example}

\theoremstyle{definition}

\newtheoremstyle{citing}
  {}
  {}
  {\itshape}
  {}
  {\bfseries}
  {.}
  {.5em}
  {\thmnote{#3}}

\theoremstyle{citing}

\newcommand{\Integer}{\mathbb{Z}}

\makeatletter
\@addtoreset{equation}{section}

\makeatother

\begin{document}

\begin{abstract}
A 2-sphere embedded in the 4-sphere invariant under a circle action is called a branched twist spin. 
A branched twist spin is constructed from a 1-knot in the 3-sphere and a pair of coprime integers uniquely. In this paper, we study, for each pair of coprime integers, if two different 1-knots yield the same branched twist spin, and prove that such a pair of 1-knots does not exist in most cases. Fundamental groups of $3$-orbifolds of cyclic type are obtained as quotient groups of the fundamental groups of the complements of branched twist spins. We use these groups for distinguishing branched twist spins.
\end{abstract}


\maketitle

\section{Introduction}

Circle actions on homotopy $4$-spheres are classified by Montgomery and Yang~\cite{MY60} and Fintushel~\cite{Fin76}. 
Later, Pao proved in~\cite{Pa78} that such homotopy $4$-spheres are the standard $4$-sphere $S^4$. 
A $2$-sphere embedded in $S^4$ that is invariant under a circle action is called a {\it branched twist spin} (cf.~\cite[\S 16.3]{Hil02}). Note that a circle embedded in $S^3$ that is invariant under a circle action is called a torus knot. A twist spun knot is an example of a branched twist spin. See Section~2.1 for more precise explanation.
In this sense, branched twist spins are the simplest class of $2$-knots in $S^4$ if one wants to understand $2$-knots geometrically.
The orbit space of a circle action on $S^4$ is the $3$-sphere $S^3$ except a trivial case and 
the image of the union of singular orbits by the orbit map constitutes an embedded circle, that is a $1$-knot, in $S^3$.
Conversely, we may recover a branched twist spin from a $1$-knot in $S^3$ with a pair of coprime integers $m$ and $n$ with $n>0$ uniquely. 
For a given knot $K$ in $S^3$ and these integers $m$ and $n$, we denote the corresponding branched twist spin by $K^{m,n}$.

In this paper, we distinguish two knots up to orientation of the ambient spaces.
Two knots $K_1$ and $K_2$ in $S^3$ are said to be {\it equivalent} if there exists a (not necessary orientation-preserving) diffeomorphism $h:S^3\to S^3$ with $h(K_1)=K_2$.
Two $2$-knots $\hat K_1$ and $\hat K_2$ are said to be  {\it equivalent} if there exists a (not necessary orientation-preserving) diffeomorphism $h:S^4\to S^4$ with $h(\hat K_1)=\hat K_2$.

The aim of this paper is to understand if we can obtain the same branched twist spin from different $1$-knots in $S^3$.
This problem is studied by the first author in~\cite{F17, F17a, F22}
using knot determinants, irreducible $SL(2,\mathbb C)$-metabelian representations, and $SL(2,\mathbb Z_3)$-representations of 
$\pi_1(S^4\setminus K^{m,n})$, respectively.
Especially, it is proved in~\cite{F22} that, for two branched twist spins $K_1^{m_1,n_1}$ and $K_2^{m_2,n_2}$
constructed from two $1$-knots $K_1$ and $K_2$, 
if $m_1\ne m_2$ then $K_1^{m_1,n_1}$ and $K_2^{m_2,n_2}$ are not equivalent.
Note that $K^{m,n}$ and $K^{m,n+2m}$ are equivalent since $K^{m,n+m}$ is obtained from $K^{m,n}$ by the Gluck surgery along $K^{m,n}$, see the explanation in the proof of~\cite[Corollary 3.3]{F18}.
By reversing the orientation of $S^4$ if necessary and applying the isomorphism between $K^{m,n}$ and $K^{m,n+2m}$, we can assume that $(m,n)$ is either a pair of coprime positive integers with $n<m$,  $(0,1)$ or $(1,1)$.

Our main result is the following.

\begin{theorem}\label{thm01}
Let $m\geq 3$, and let $K_1$ and $K_2$ be prime and non-torus knots.
Then, $K_1^{m,n}$ and $K_2^{m,n}$ are equivalent if and only if $K_1$ and $K_2$ are.
\end{theorem}

For composite knots and torus knots, see Remark~\ref{remark1}.
If $(m,n)=(0,1)$ then $K^{m,n}$ is a spun knot and $\pi_1(S^4\setminus K^{m,n})$ is isomorphic to $\pi_1(S^3\setminus K)$. Therefore, by the positive answer to the complement conjecture~\cite{GL89},
if $K_1$ and $K_2$ are prime and
not equivalent then $K_1^{0,1}$ and $K_2^{0,1}$ are also not equivalent.
If $(m,n)=(1,1)$ then $K^{m,n}$ is always the trivial $2$-knot in $S^4$~\cite{Zee65}. 

In the case $m=2$, that is $(m,n)=(2,1)$, the following result holds.

\begin{theorem}\label{thm02}
Let $m=2$, and let $K_1$ and $K_2$ be prime knots.
Assume that $K_i$ is either a satellite knot or Conway reducible for $i=1,2$.
Then, $K_1^{m,n}$ and $K_2^{m,n}$ are equivalent if and only if $K_1$ and $K_2$ are.
\end{theorem}

Recall that a $1$-knot is said to be {\it Conway reducible} if it admits an essential Conway sphere, that is, an essential four-punctured sphere in its complement. 
For knots without essential tori and essential Conway spheres, see Remark~\ref{remark2}.

To prove the above theorems, we use known results about $3$-orbifolds of cyclic type
in~\cite{Tak91, BP01}.
The point is that the quotient group of $\pi_1(S^4\setminus K^{m,n})$ by its center can be regarded as
the fundamental group of a $3$-orbifold of cyclic type 
in most cases (Lemma~\ref{lemma41}).
We also use the fact that the fiber of $K^{m,n}$ is obtained as the $m$-fold cyclic branched cover of $S^3$ along $K$ with removing an open $3$-ball~\cite{Zee65, Pa78} and its fundamental group is the commutator subgroup
of  $\pi_1(S^4\setminus K^{m,n})$.

In Section~2, we give brief explanations and introduce known results about branched twist spins and $3$-orbifolds needed to prove our theorems.
Proofs are given in Section~3.

The authors would like to express their sincere gratitude to Makoto Sakuma for pointing out a gap in the proof of Lemma 3.1 and telling us how to fill it.
They would also like to thank Jonathan A. Hillman, Yuya Koda, Kokoro Tanaka and Yuta Taniguchi for their helpful suggestions and precious comments.
The second author is supported by 
JSPS KAKENHI Grant Numbers JP19K03499, JP23K03098 and JP23H00081, JSPS-VAST Joint Research Program, Grant number JPJSBP120219602, and Keio University Academic Development Funds for Individual Research.

\section{Preliminaries}

\subsection{Branched twist spins}

Circle actions on homotopy $4$-spheres are classified by Montgomery and Yang~\cite{MY60} and Fintushel~\cite{Fin76}, and Pao proved in~\cite{Pa78} that such homotopy $4$-spheres are the standard $4$-sphere $S^4$. 
See for instance~\cite{Bre72} for basic terminologies of circle actions.


\begin{theorem}[\cite{Pa78} Theorem 4]
The weak equivalence classes of effectively locally smoothly circle actions on $S^4$ are classified into
 four cases: {\rm (1)} $\{D^3\}$, {\rm (2)} $\{S^3\}$, {\rm (3)} $\{S^3, m\}$, and {\rm (4)}~$\{(S^3,K), m,n\}$.
\end{theorem}

The notation $\{(S^3,K), m,n\}$ means the circle action satisfies the following properties:
\begin{itemize}
\item The orbit space is $S^3$.
\item The image of the union of singular orbits by the orbit map constitutes a 1-knot $K$ in the orbit space $S^3$.
\item $m$ and $n$ are coprime integers more than $1$ and the types of exceptional orbits are $\mathbb{Z}/m\mathbb{Z}$ and  $\mathbb{Z}/n\mathbb{Z}$.
\end{itemize}
Here, an exceptional orbit is said to be of $\mathbb{Z}/m\mathbb{Z}$-type 
if its isotropy group is isomorphic to $\mathbb{Z}/m\mathbb{Z}$.
Let $E_m$ and $E_n$ be the set of exceptional orbits of $\mathbb{Z}/m\mathbb{Z}$-type 
and $\mathbb{Z}/n\mathbb{Z}$-type, respectively, and $F$ be the fixed point set.
It is known that the union $E_n \cup F$ is diffeomorphic to $S^2$ in $S^4$, which is called the $(m,n)$-{\it branched twist spin} of $K$ and denoted as $K^{m,n}$.
Let $E_m^*$, $E_n^*$ and $F^*$ be the images of $E_m$, $E_n$ and $F$ by the orbit map, respectively.
The images $E_m^*$ and $E_n^*$ are arcs connecting the two points $F^*$,
and the union $E_m^*\cup E_n^*\cup F^*$ is the $1$-knot $K$.

If $n=1$ then $E_n$ becomes a union of free orbits. This is the case $\{S^3, m\}$.
Choose an arc $E_n^*$ in the orbit space $S^3$ connecting the two point $F^*$ so that 
the union $E_m^*\cup E_n^*\cup F^*$ constitutes a $1$-knot $K$. Then  the union $E_n \cup F$ is again diffeomorphic to $S^2$ in $S^4$, which is the so-called {\it $m$-twist spun knot} of $K$~\cite{Zee65}.
We call this the $(m,1)$-{\it branched twist spin} of $K$ and denote it as $K^{m,1}$.

If $(m,n)=(1,1)$ then $E_m\cup E_n$ becomes a union of free orbits,
which is the case $\{S^3\}$. Choose arcs $E_m^*$ and $E_n^*$ in the orbit space $S^3$ connecting the two point $F^*$ so that $E_m^*\cup E_n^*\cup F^*$ constitutes a $1$-knot $K$.
The union $E_n \cup F$ is a $1$-twist spun knot of $K$ and it is a trivial $2$-knot in $S^4$ as mentioned in~\cite{Zee65}.
We call this the $(1,1)$-{\it branched twist spin} of $K$ and denote it as $K^{1,1}$.

The notation $\{D^3\}$ is the circle action whose fixed point set is the trivial $2$-knot $\hat K$ in $S^4$.
The circle action yields the trivial fiberation in $S^4\setminus \hat K$ over $S^1$.
The orbit space is the $3$-ball $D^3$. 
Choose a properly embedded arc $k$ in $D^3$. Embed $D^3$ into $S^3$ and shrink $S^3\setminus \text{Int} D^3$ to a point. Then the embedded arc $k$ becomes a $1$-knot $K$ in $S^3$.
The preimage of $k$ by the orbit map is a $2$-knot in $S^4$,
which is the so-called {\it spun knot} of $K$~\cite{Art25}.
We call this the $(0,1)$-{\it branched twist spin} of $K$ and denote it as $K^{0,1}$.

In~\cite{Pa78}, Pao constructed the complement of $K^{m,n}$ by setting the $m$-fold cyclic branched cover of $S^3$ along $K$ with removing an open ball to be a fiber and the monodromy to be periodic, shifting the sheets of the cyclic branched cover by $n$, and then proved that $K^{m,n}$ is a $2$-knot in $S^4$.
His results are summarized as follows:

\begin{theorem}[Pao \cite{Pa78}]\label{thmpao}
If $m\geq 1$, then $K^{m,n}$ is a fibered 2-knot in $S^4$ whose fiber is diffeomorphic to the $m$-fold cyclic branched cover of $S^3$ along $K$ with removing an open ball and whose monodromy is periodic, shifting the sheets of the cyclic branched cover by $n$.
\end{theorem}

Next, we introduce a presentation of the fundamental group of $K^{m,n}$ given in~\cite{Plo84} (see also~\cite{F17}).
Let $\langle x_1, \ldots, x_l \mid r_1, \ldots, r_l\rangle$ be a Wirtinger presentation of a $1$-knot $K$ in $S^3$.
The union of free orbits in $S^4\setminus K^{m,n}$ is a dense subset of $S^4\setminus K^{m,n}$.
Let $h$ denote the element of $\pi_1(S^4\setminus K^{m,n})$ corresponding to a free orbit.
Then $\pi_1(S^4 \setminus K^{m,n})$ is presented as
\begin{equation}\label{knotgroup}
\pi_1(S^4 \setminus K^{m,n}) \cong \langle x_1, \ldots, x_l, h \mid r_1, \ldots, r_l,  x_1hx^{-1}_1h^{-1}, \ldots, x_lhx^{-1}_lh^{-1}, x_1^mh^{\beta}\rangle,
\end{equation}
where $\beta$ is an integer satisfying $n\beta \equiv 1\ ({\rm mod}\ m)$.

It will be proved in Lemma~\ref{lemma41} that if $K$ is neither a trivial knot nor a torus knot 
and the $m$-fold cyclic branched cover of $S^3$ along $K$ is aspherical, then 
the center of $\pi_1(S^4\setminus K^{m.n})$ is $\langle h\rangle$.
In this case, the quotient group of the group \eqref{knotgroup} by its center 
is 
\begin{equation*}
\pi_1(S^4\setminus K^{m,n})/\langle h\rangle \cong 
\langle x_1, \ldots, x_l \mid r_1, \ldots, r_l, x_1^m \rangle.
\end{equation*}
This group is regarded as the fundamental group of the $3$-orbifold of cyclic type with underlying space $S^3$ and ramification locus $K$ of order $m$.

\subsection{$3$-orbifolds}

An orbifold is a connected and separable metric space locally homeomorphic to the quotient space of $\mathbb R^n$ by a finite group action. 
This notion is introduced by Satake~\cite{Sat56} and renamed by W.~Thurston~\cite[Chapter 13]{Thu77}. 
A $3$-dimensional orbifold ($3$-orbifold for short) whose ramification locus is the disjoint union of simple closed curves is called a $3$-orbifold of {\it cyclic type} in~\cite{DM84, BP01} and called a {\it $3$-branchfold} in~\cite{Tak88}. 
A $3$-orbifold is said to be {\it sufficiently large} (or {\it Haken}) if there exists an orientable and incompressible $2$-suborbifold that is not boundary parallel.
See for instance~\cite{Thu77, Tak91, TY99, TY02, BP01, BMP03, BLP05} for basics of $3$-orbifolds.

The fundamental group $\pi_1^{orb}\mathcal O$ of a $3$-orbifold $\mathcal O$ is defined to be the deck transformation group of the universal covering of $\mathcal O$. 
Let $L$ be a $1$-link in $S^3$, $m\geq 2$, and $\mathcal O(L,m)$ be the $3$-orbifold of cyclic type with underlying space $S^3$ and ramification locus $L$ of order $m$.
Let $M_m(L)$ denote the $3$-manifold obtained as the $m$-fold cyclic branched cover of $S^3$ along $L$.
The covering map $M_m(L)\to \mathcal O(L,m)$ induces a monomorphism 
$\pi_1(M_m(L))\to \pi_1^{orb}(\mathcal O(L,m))$ (cf.~\cite[Proposition 2.5]{BMP03}).


In~\cite{Tak91}, Takeuchi studied when two $3$-orbifolds with the isomorphic fundamental groups are homeomorphic as orbifolds, and obtained the next theorem as one of the consequences of his arguments.

\begin{theorem}[\cite{Tak91} Theorem 8.1]\label{thmTak91}
Let $L_1$ and $L_2$ be prime links in $S^3$ such that $\mathcal O(L_1,r)$ and $\mathcal O(L_2,r)$ are sufficiently large for some $r\geq 2$.
Then, $L_1$ and $L_2$ are equivalent if and only if $\pi_1^{orb}(\mathcal O(L_1,m))$ and $\pi_1^{orb}(\mathcal O(L_2,m))$ are isomorphic for some $m\geq 2$.
\end{theorem}

Note that if  $\mathcal O(K,r)$ is sufficiently large for some $r\geq 2$ then it is for any $r\geq 2$.
A sufficient condition for a prime link $L$ to satisfy that $\mathcal O(L,r)$ is sufficiently large for some $r\geq 2$ is given in~\cite[Proposition~8.2]{Tak91}.




A complete proof of Thurston's orbifold theorem for $3$-orbifolds of cyclic type is given by Boileau and Porti in~\cite{BP01}.
As a corollary of this theorem with the classification of orientable closed Euclidean and spherical $3$-orbifolds (cf.~\cite{BS87, Dun88}), the following result is obtained.

\begin{theorem}[\cite{BP01} Corollary~5]\label{thmBP5}
Let $M$ be a compact, orientable, irreducible $3$-manifold and $L$ be a hyperbolic link in $M$.
Then, for $m\geq 3$, any $m$-fold cyclic branched cover of $M$ along $L$ admits a hyperbolic structure,
except when $m=3$, $M=S^3$ and $L$ is the figure-eight knot.
\end{theorem}

\section{Proofs}

Let 
$\mathcal{T}$ be the set of torus knots,
$\mathcal{H}$ be the set of hyperbolic knots,
and $\mathcal{S}$ be the set of satellite knots in $S^3$.
The trivial knot does not belong to any of these sets in this paper.
Let $M_m(K)$ denote the $m$-fold cyclic branched cover of  $S^3$ along $K$.

\begin{lemma}\label{lemma41}
Suppose $m\geq 2$. 
If $K$ is non-trivial and $K\not\in \mathcal T$
  and $M_m(K)$ is aspherical,
then the center of $\pi_1(S^4\setminus K^{m,n})$ is isomorphic to $\mathbb Z$, which is generated by $h$ in~\eqref{knotgroup}, and, in particular,
\[
  \pi_1(S^4\setminus K^{m,n})/\langle h\rangle\cong \pi_1^{orb}(\mathcal O(K,m))
\]
is an invariant of $K^{m,n}$. If $K$ is trivial then 
$\pi_1(S^4\setminus K^{m,n})\cong\mathbb Z$ and
$\pi_1(S^4\setminus K^{m,n})/\langle h\rangle=1$.
\end{lemma}

\begin{proof}
If $K$ is a trivial knot then the fiber of $K^{m,n}$ is a $3$-ball for any $m,n$ from the fact that $M_m(K)$ is $S^3$ again. 
Thus $K^{m,n}$ is the trivial 2-knot and $\pi_1(S^4 \setminus K^{m,n})\cong\Integer$.


Let $K$ be a non-trivial knot and assume that $K\not\in\mathcal T$ and $M_m(K)$ is aspherical.
First we prove that the center of $\pi_1^{orb}(\mathcal O(K,m))$ is trivial.

Assume that there exists a non-trivial element $\gamma$ in the center of $\pi_1^{orb}(\mathcal O(K,m))$ of finite order.
Since $M_m(K)$ is aspherical, the universal covering space of $\mathcal O(K,m)$ is $\mathbb R^3$ by the Orbifold Theorem~\cite{BP01, BMP03}.
The action of $\gamma$ to the universal covering space $\mathbb R^3$ is periodic and hence conjugate to an element of $SO(3)$ in Diff${}^+(\mathbb R^3)$ by~\cite[Theorem 4]{MY84}. 
In particular, it has a fixed point.
The action of $\gamma$ induces a self diffeomorphism
$g_\gamma$ of $M_m(K)$ having a fixed point.
Since the action of $\gamma$ is conjugate to a non-trivial element of $SO(3)$, 
$g_\gamma$ is non-trivial in $\pi_1^{orb}(\mathcal O(K,m))/\pi_1(M_m(K)) \cong \mathbb Z/m\mathbb Z$.
However, since $\gamma$ is in the center of  $\pi_1^{orb}(\mathcal O(K,m))$, the automorphism 
of $\pi_1(M_m(K),b)$ induced by $\gamma$ is trivial, where $b$ is a fixed point of $g_\gamma$ on $M_m(K)$.
Then $g_\gamma$ should be trivial by~\cite[Theorem 1.1]{CR72}. This is a contradiction.

Assume that there exists a non-trivial element $\gamma$ in the center of $\pi_1^{orb}(\mathcal O(K,m))$ of infinite order. 
Then $\gamma^m$ is a non-trivial element in the center of $\pi_1(M_m(K))$.
Hence, by the positive answer to the Seifert fiber space conjecture~\cite{Gab92, CJ94} (cf.~\cite{Pre14}),
$M_m(K)$ is a Seifert fibered space.
By~\cite{MS86}, the action of $\mathbb Z/m\mathbb Z$ on $M_m(K)$ is conjugate to a fiber-preserving action.
Since the action of $\mathbb Z/m\mathbb Z$ on $\pi_1(M_m(K))$ sends $\gamma^m$ to itself,
the action of $\mathbb Z/m\mathbb Z$  on $M_m(K)$ preserves the orientations of the fibers.
This implies that a fiber in $M_m(K)$ intersecting the preimage $\tilde K$
of $K$ should coincide with $\tilde K$ itself
since $\tilde K$ is the fixed point set of the action of $\mathbb Z/m\mathbb Z$.
Thus, $\mathcal O(K,m)=M_m(K)/(\mathbb Z/m\mathbb Z)$ is a Seifert fibered orbifold
consisting of only circle fibers one of which is $K$.
In particular, $\mathcal O(K,m)\setminus K$ is a Seifert fibered space.
This contradicts the assumption that $K\not\in\mathcal T$.

Thus, we conclude that the center of $\pi_1^{orb}(\mathcal O(K,m))$ is trivial. 

Set $G=\pi_1(S^4\setminus K^{m,n})$ and let $Z(G)$ denote the center of $G$.
By~\eqref{knotgroup}, the subgroup of $G$ generated by $h$ is a subgroup of $Z(G)$. Each element in $Z(G)$ is written in the form $xh^\delta$, where $x$ is a product of the generators $x_1,\ldots,x_l$ in~\eqref{knotgroup} and $\delta\in\mathbb Z$. 
In particular, if $xh^\delta$ is in $Z(G)$ then $x$ is also. 

Let $y$ be a non-trivial element in $Z(G)$ that is given as a product of the generators $x_1,\ldots,x_l$ in~\eqref{knotgroup} but cannot be presented as a product of $x_1^{m},\ldots,x_l^{m}$.
Assume that $yx_i y^{-1}x_i^{-1}$ is trivial in $G$ for $i=1,\ldots,l$.
Then, there exists a sequence of Tietze transformations that changes $yx_i y^{-1}x_i^{-1}$ to $1$, possibly including the word $h$ 
in the middle.
Since $h \in Z(G)$, we obtain a sequence of Tietze transformations that changes $yx_i y^{-1}x_i^{-1}$ to $1$, without $h$,
up to adding and deleting the words $x_1^{m},\ldots,x_l^{m}$ anywhere in the middle.
This means that $yx_i y^{-1}x_i^{-1}$ is trivial in the orbifold group $\pi_1^{orb}(\mathcal O(K,m))$ for all $i=1,\ldots,l$, that is, $y$ is in the center of $\pi_1^{orb}(\mathcal O(K,m))$.
%
Since the center of $\pi_1^{orb}(\mathcal O(K,m))$ is trivial, $y$ is trivial in $\pi_1^{orb}(\mathcal O(K,m))$.
This contradicts the assumption that $y$ is non-trivial.

Now we see that any element in $Z(G)$ given as a product of the generators $x_1,\ldots,x_l$ is presented as a product of $x_1^{m},\ldots,x_l^{m}$. 
This means that $Z(G)$ is generated by $h$. Setting $h$ in the presentation~\eqref{knotgroup} to be trivial, we obtain 
$G/Z(G)\cong \langle x_1,\ldots, x_l \mid r_1,\ldots, r_l,\, x_1^{m}\rangle$. 
Thus the assertion holds.
\end{proof}

In the following, we prove three propositions concerning the distinction of two branched twist spins.

\begin{proposition}\label{prop1}
Suppose $m\geq 2$, $K_1$ is a trivial knot and $K_2$ is non-trivial. 
Then $K_1^{m,n}$ and $K_2^{m,n}$ are not equivalent.
\end{proposition}

\begin{proof}
Assume that $K_1^{m,n}$ and $K_2^{m,n}$ are equivalent. 
Then $\pi_1(S^4\setminus K_1^{m,n})$ and $\pi_1(S^4\setminus K_2^{m,n})$ are isomorphic
and hence their commutator subgroups are also isomorphic.
The fiber of $K_i^{m,n}$, for $i=1,2$, is the $m$-fold cyclic branched cover $M_m(K_i)$ of $S^3$ along $K_i$ with removing an open $3$-ball, and hence the commutator subgroup of $\pi_1(S^4\setminus K_i^{m,n})$ is isomorphic to $\pi_1(M_m(K_i))$. Thus $\pi_1(M_m(K_1))$ and $\pi_1(M_m(K_2))$ are isomorphic.
Since $M_m(K_1)$ is $S^3$, $M_m(K_2)$ is also by the positive solution to the Poincar\'{e} conjecture~\cite{Per02, Per03a, Per03b}.
However, since $M_m(K_2)$ is the $m$-fold cyclic branched cover of $S^3$ along a non-trivial knot $K_2$, it cannot be homeomorphic to $S^3$ due to the positive solution to the Smith conjecture~\cite{MB84}. 
This is a contradiction.
\end{proof}



\begin{proposition}\label{prop3}
Let $m\geq 3$, and let $K_1$ be a hyperbolic knot and $K_2$ be a prime, satellite knot.
Then $K_1^{m,n}$ and $K_2^{m,n}$ are not equivalent.
\end{proposition}

\begin{proof}
Let $T^2$ be an innermost essential torus in $S^3\setminus K_2$.
If the kernel of $T^2\to\mathcal O(K_2,m)$ has a non-trivial element, 
by the Loop Theorem for $3$-orbifolds, there exists a discal $2$-suborbifold $D$ properly embedded in the solid torus bounded by $T^2$ such that $\partial D$ is not null-homotopic in $T^2$. Since $D$ intersects $K_2$ at one point transversely, $K_2$ is a composite knot.
This contradicts the assumption that $K_2$ is prime.
Hence $T^2$ is incompressible in $\mathcal O(K_2,m)$
and 
$\pi_1^{orb}(\mathcal O(K_2,m))$ contains a subgroup $\mathbb Z\oplus \mathbb Z$.

Suppose that $m\geq 3$ and if $m=3$ then $K_1$ is not a figure-eight knot.
If $K_1^{m,n}$ and $K_2^{m,n}$ are equivalent then $\pi_1^{orb}(\mathcal O(K_1,m))$ and $\pi_1^{orb}(\mathcal O(K_2,m))$ are isomorphic by Lemma~\ref{lemma41}, and hence $\pi_1^{orb}(\mathcal O(K_1,m))$ also contains a subgroup $\mathbb Z\oplus\mathbb Z$.
However, the orbifold $\mathcal O(K_1,m)$ is hyperbolic and compact by Theorem~\ref{thmBP5}.
Hence $\pi_1^{orb}(\mathcal O(K_1,m))$ does not contain a subgroup $\mathbb Z\oplus \mathbb Z$.
This is a contradiction.

Suppose $m=3$ and $K_1$ is a figure-eight knot. The manifold $M_3(K_1)$ is the so-called Hantzsche-Wendt manifold.
It is known in~\cite{LS84} that $\pi_1(M_3(K_1))$ contains 
$\mathbb Z\oplus\mathbb Z\oplus\mathbb Z$ as a subgroup.
Since the covering map $\pi_1(M_3(K_1))\to \pi_1^{orb}(\mathcal O(K_1,3))$ is a monomorphism,
$\pi_1^{orb}(\mathcal O(K_1,3))$ contains a subgroup $\mathbb Z\oplus\mathbb Z\oplus\mathbb Z$.
If
$K_1^{m,n}$ and $K_2^{m,n}$ are equivalent then $\pi_1^{orb}(\mathcal O(K_1,3))$ and $\pi_1^{orb}(\mathcal O(K_2,3))$ are isomorphic by Lemma~\ref{lemma41}, and hence
$\pi_1^{orb}(\mathcal O(K_2,3))$ also contains a subgroup $\mathbb Z\oplus\mathbb Z\oplus\mathbb Z$.
Since $\pi_1(M_3(K_2))\to \pi_1^{orb}(\mathcal O(K_2,3))$ is a monomorphism,
$\pi_1(M_3(K_2))$ also contains a subgroup $\mathbb Z\oplus\mathbb Z\oplus\mathbb Z$.
The closed $3$-manifolds whose fundamental groups contain  a subgroup $\mathbb Z\oplus\mathbb Z\oplus\mathbb Z$ are classified in~\cite{LS84}. Especially, it is proved that if such a manifold $M$ is orientable and
irreducible then
there is a finite covering map from the $3$-torus $T^3$ to $M$. Therefore, there exists a finite covering map
$T^3\to M_3(K_2)$. Let $\phi :T^3\to \mathcal O(K_2,3)$ be the composition of the covering maps $T^3\to M_3(K_2)$ and $M_3(K_2)\to \mathcal O(K_2,3)$
and let $\tilde T^2$ be a torus in $T^3$ such that $\phi(\tilde T^2)=T^2$. 

Assume that $\tilde T^2$ has a compressing disk $D$. Since $\phi(\partial D)$ in $T^2\subset \mathcal O(K_2,3)$ is null-homotopic, the Loop Theorem for $3$-orbifolds implies that there exists a discal $2$-suborbifold in $\mathcal O(K_2,3)$ bounded by  $\phi(\partial D)$.
In particular, $T^2$ is compressible in $\mathcal O(K_2,3)$, which contradicts the assumption that $K_2$ is prime. Therefore, $\tilde T^2$ is incompressible. 

Since $\tilde T^2$ is a torus in $T^3$ obtained as a component of the preimage of $T^2$ by $\phi$, it separates $T^3$ into two pieces. However, $T^3$ does not contain such an essential torus, which is a contradiction.
\end{proof}


\begin{proposition}\label{thm4}
Suppose $m\geq 2$.
For $i=1,2$, let $K_i$ be a prime, non-torus knot in $S^3$ such that 
$\mathcal O(K_i,m)$ is sufficiently large for some $m\geq 2$.
Then,  $K_1^{m,n}$ and $K_2^{m,n}$ are equivalent if and only if $K_1$ and $K_2$ are. 
In particular, if $K_1, K_2\in\mathcal S$ are prime, then $K_1^{m,n}$ and $K_2^{m,n}$ are equivalent if and only if $K_1$ and $K_2$ are.
\end{proposition}


\begin{proof}[Proof  of Proposition~\ref{thm4}]
If $K_1^{m,n}$ and $K_2^{m,n}$ are equivalent then $\pi_1^{orb}(\mathcal O(K_1,m))$ and $\pi_1^{orb}(\mathcal O(K_2,m))$ are isomorphic by Lemma~\ref{lemma41}.
Hence, by Theorem~\ref{thmTak91}, $K_1$ and $K_2$ are equivalent.
If $K_1, K_2\in\mathcal S$ are prime, $\mathcal O(K_1,m)$ and $\mathcal O(K_2,m)$ contain incompressible tori as explained in the first paragraph in the proof of Proposition~\ref{prop3}. Hence the latter assertion holds.
\end{proof}

\begin{proof}[Proof  of Theorem~\ref{thm01}]
Due to Propositions~\ref{prop1}, \ref{prop3} and~\ref{thm4}, it is enough to show the case where $K_1, K_2\in\mathcal H$.
Suppose that $K_1, K_2\in\mathcal H$. If $K_1^{m,n}$ and $K_2^{m,n}$ are equivalent then
$\pi_1^{orb}(\mathcal O(K_1,m))$ and $\pi_1^{orb}(\mathcal O(K_2,m))$ are isomorphic by Lemma~\ref{lemma41}.
If either $m\geq 4$ or both of $K_1$ and $K_2$ are not a figure-eight knot, then, by Theorem~\ref{thmBP5},  $M_m(K_1)$ and $M_m(K_2)$ are hyperbolic and the assertion follows from the Mostow rigidity 
(cf.~\cite[Theorem 6.9]{BMP03}).
Suppose that $m=3$, $K_1$ is a figure-eight knot and $K_2$ is not.
The Hantzsche-Wendt manifold $M_3(K_1)$ admits a riemannian metric with sectional curvature being constantly zero~\cite{Wo67, LS84}. Therefore, $\pi_1(M_3(K_1))$ is not isomorphic to $\pi_1(M_3(K_2))$ since $M_3(K_2)$ is hyperbolic by Theorem~\ref{thmBP5}.
Hence $K_1^{m,n}$ and $K_2^{m,n}$ are not equivalent.
\end{proof}

\begin{remark}\label{remark1}
Composite knots and torus knots are excluded in Theorem~\ref{thm01}.
Note that the fundamental groups of the $3$-orbifolds having ramification loci along the granny knot and the square knot with order $2$ are isomorphic. 
Therefore, we cannot use the argument in this paper for the pair of these composite knots if $m=2$.
It is possible to distinguish a torus knot from some non-torus knot by their branched twist spins if $m \geq 3$, see Proposition~\ref{prop4} below.
If both of knots $K_1$ and $K_2$ are torus knots, then it is proved in~\cite{Hil23} that, for $m \geq 2$, $\pi_1(S^4\setminus K_1^{m,n})\cong \pi_1(S^4\setminus K_2^{m, n})$ implies that $K_1$ and $K_2$ are equivalent.
\end{remark}

\begin{proposition}\label{prop4}
Let $m\geq 3$, and let $K_1$ be a torus knot and $K_2$ be a non-torus knot.
Assume that if $K_2$ is prime then it is not a satellite knot.
Then $K_1^{m,n}$ and $K_2^{m,n}$ are not equivalent.
\end{proposition}

\begin{proof}
Suppose that $K_1$ is a $(p,q)$-torus knot, where $p,q\geq 2$.
By Proposition~\ref{prop1}, we can assume that $K_2$ is non-trivial.
Suppose that $K_2$ is a composite knot. Assume that $K_1^{m,n}$ and $K_2^{m,n}$ are equivalent.
For each $i=1,2$, the fiber of $K_i^{m,n}$ is obtained from $M_m(K_i)$ by removing an open $3$-ball.
Since $\pi_1(M_m(K_1))$ and $\pi_1(M_m(K_2))$ are the commutator subgroups of $\pi_1(S^4\setminus K_1^{m,n})$ and $\pi_1(S^4\setminus K_2^{m,n})$, respectively, they are isomorphic. However, $\pi_1(M_m(K_2))$ is a free product of non-trivial groups while $\pi_1(M_m(K_1))$ is not due to the positive answer to the Kneser Conjecture~\cite{Sta71, Hem76}.
 This is a contradiction.

Suppose that $K_2\in \mathcal H$. Assume that if $m=3$ then $K_2$ is not a figure-eight knot.
For each $i=1,2$, the fiber of $K_i^{m,n}$ is obtained from $M_m(K_i)$ by removing an open $3$-ball.
Since $M_m(K_1)$ is Seifert fibered and $M_m(K_2)$ is hyperbolic,
$\pi_1(M_m(K_1))$ and $\pi_1(M_m(K_2))$ are not isomorphic.
Hence $K_1^{m,n}$ and $K_2^{m,n}$ are not equivalent. 
Suppose that $m=3$ and $K_2$ is a figure-eight knot.
It is known that the Hantzsche-Wendt manifold $M_3(K_2)$ satisfies
$H_1(M_3(K_2))=(\mathbb Z/4\mathbb Z)\oplus (\mathbb Z/4\mathbb Z)$, see~\cite{LS84, Hem87}.
On the other hand, $M_3(K_1)$ is the Brieskorn manifold of the singularity of the complex polynomial map $f(x,y,z)=x^p+y^q+z^3$ at the origin. 
Using the formulas in~\cite[Corollary 6.2 and Theorem 7.2]{JN83}, we can confirm that 
$|H_1(M_3(K_1))|$ cannot be $16$.
Thus, $\pi_1(M_3(K_2))$ is not isomorphic to $\pi_1(M_3(K_1))$.
Hence $K_1^{m,n}$ and $K_2^{m,n}$ are not equivalent.
\end{proof}

\begin{proof}[Proof  of Theorem~\ref{thm02}]
Suppose that $K_1^{m,n}$ and $K_2^{m,n}$ are equivalent.
Since $K_1$ and $K_2$ are prime, non-trivial and non-torus knots and 
$M_m(K_1)$ and $M_m(K_2)$ are aspherical,
 $\pi_1^{orb}(\mathcal O(K_1,m))$ and $\pi_1^{orb}(\mathcal O(K_2,m))$ are isomorphic by Lemma~\ref{lemma41}.

For $i=1,2$, if $K_i$ is a prime, satellite knot then $\mathcal O(K_i,r)$ contains an incompressible torus as mentioned in the proof of Proposition~\ref{prop3}, and hence $\mathcal O(K_i,r)$ is sufficiently large for $r\geq 2$.
If $K_i$ is Conway reducible, then $\mathcal O(K_i,r)$ is also sufficiently large for $r\geq 2$.
Hence, since $\pi_1(\mathcal O(K_1,m))$ and $\pi_1(\mathcal O(K_2,m))$ are isomorphic,
$K_1$ and $K_2$ are equivalent by Theorem~\ref{thmTak91}.
\end{proof}

\begin{remark}\label{remark2}
Knots without essential tori and essential Conway spheres are excluded in Theorem~\ref{thm02}. It is a difficult problem to determine the class of knots determined by their double branched covers, see~\cite[Problem 3.25]{Kir97}.
There are several knots that are not determined by their double branched covers, see for instance~\cite{JP21}.
It is well-known that a mutation pair of non-equivalent Montesinos knots $K_1$ and $K_2$ has the same double branched covers. However, $\pi_1(\mathcal O(K_1,2))$ and 
 $\pi_1(\mathcal O(K_2,2))$ are different by Theorem~\ref{thmTak91}.
This means that the distinction by $3$-orbifold groups works for such a pair though that by double branched covers does not.
\end{remark}

\end{document}